\documentclass[oneside, article, a4paper]{memoir}

\usepackage[T1]{fontenc}
\usepackage[utf8]{inputenc}
\usepackage[a4paper, left=3.5cm, right=3.5cm, top=3cm, bottom=3.5cm]{geometry}
\usepackage{enumitem}
\usepackage{amsmath}
\usepackage{amssymb}
\usepackage{amsthm}
\usepackage{stmaryrd}
\usepackage{microtype}
\usepackage{parskip}
\usepackage{tikz-cd}
\usepackage{enumitem}
\usepackage{xcolor}
\usepackage{lmodern}
\usepackage{todonotes}

\author{Lukas Heidemann}
\title{Frames in Pretriangulated Dg-Categories}
\date{May 29, 2020}

\setsecnumdepth{subsection}
\counterwithout{section}{chapter}
\setsecheadstyle{\center\scshape}

\newtheorem{thm}{Theorem}[section]
\newtheorem{prop}[thm]{Proposition}
\newtheorem{lem}[thm]{Lemma}

\theoremstyle{definition}
\newtheorem{definition}[thm]{Definition}
\newtheorem{example}[thm]{Example}
\newtheorem{remark}[thm]{Remark}

\newcommand{\cat}[1]{\mathcal{#1}}

\newcommand{\sign}[1]{{(-1)}^{#1}}

\newcommand{\Nerve}{\operatorname{N}}
\newcommand{\dgNerve}{\operatorname{N}_{\operatorname{dg}}}

\newcommand{\fNerve}{\operatorname{N}_{\operatorname{f}}}

\newcommand{\Ho}{\operatorname{Ho}}

\DeclareMathOperator{\Sd}{Sd}

\newcommand{\Ob}{\operatorname{Ob}}
\newcommand{\Hom}[1]{\operatorname{Hom}_{#1}}
\newcommand{\Map}[1]{\operatorname{Map}_{#1}}
\newcommand{\MapR}[1]{\operatorname{Map}^{\operatorname{R}}_{#1}}

\renewcommand{\restriction}{\mathord{\upharpoonright}}

\newcommand{\seq}[1]{\langle #1 \rangle}

\newcommand{\PathAlg}{\operatorname{Path}}
\newcommand{\PathMod}{\operatorname{Cell}}

\newcommand{\op}[1]{#1^{\operatorname{op}}}

\newcommand{\Z}{\mathbb{Z}}
\newcommand{\N}{\mathbb{N}}

\newcommand{\FreeAb}[1]{\Z \cdot #1}

\newcommand{\id}{\operatorname{id}}

\newcommand{\Moore}{\operatorname{N}}

\newcommand{\Cyl}{\operatorname{Cyl}}
\newcommand{\Cone}{\operatorname{Cone}}

\newcommand{\Mod}[1]{#1\operatorname{-Mod}}

\newcommand{\Ch}{\operatorname{Ch}}

\newcommand{\ChCh}{\operatorname{Ch}}
\newcommand{\Ab}{\operatorname{Ab}}

\newcommand{\repr}[1]{Y#1}

\newcommand{\Ainfty}{\mathcal{A}_\infty}

\newcommand{\defn}[1]{\textit{#1}}

\newcommand{\dgCatAlg}[1]{#1^\text{alg}}

\nobibintoc

\begin{document}

\maketitle

\begin{abstract}
  Triangulated categories arising in algebra can often be described as the
  homotopy category of a pretriangulated dg-category, a category enriched in
  chain complexes with a natural notion of shifts and cones that is accessible
  with all the machinery of homological algebra.
  Dg-categories are algebraic models of $\infty$-categories and thus fit into a
  wide ecosystem of higher-categorical models and translations between them. In
  this paper we describe an equivalence between two methods to turn a pretriangulated
  dg-category into a quasicategory.

  The dg-nerve of a dg-category is a quasicategory whose simplices are coherent
  families of maps in the mapping complexes. In contrast, the cycle category of
  a pretriangulated category forgets all higher-degree elements of the mapping
  complexes but becomes a cofibration category that encodes the homotopical
  structure indirectly. This cofibration category then has an associated
  quasicategory of frames in which simplices are Reedy-cofibrant resolutions.

  For every simplex in the dg-nerve of a pretriangulated dg-category we
  construct such a Reedy-cofibrant resolution and then prove that this
  construction defines an equivalence of quasicategories which is natural up to
  simplicial homotopy. Our construction is explicit enough for calculations and
  provides an intuitive explanation of the resolutions in the quasicategory of
  frames as a generalisation of the mapping cylinder.
\end{abstract}

\tableofcontents*

\section{Introduction}

Homotopical structure arises in various subfields of mathematics, which has
lead to the development of a multitude of approaches to abstract homotopy
theory. These approaches make different tradeoffs and are each accessible to
respective flavours of either concrete calculation or high-level constructions.
By considering models of a single theory in a variety of these frameworks,
manipulations can be performed in the context where they are most natural.

Differentially graded categories are categories that are enriched in chain
complexes.  The $0$-dimensional cycles in the mapping complexes of a
dg-category $\cat{C}$ are the maps of the underlying ordinary category
$\cat{C}_0$, while the $1$-dimensional chains function as homotopies and the
higher-degree chains as coherence data. A pretriangulated dg-category
additionally satisfies some representability properties that allow the
construction of cones and shifts, making the underlying category into a
triangulated category. Many triangulated categories in algebra are induced in
this way by a pretriangulated dg-category~\cite{enhanced-triangulated}.

Coherent families of higher-degree maps in the mapping complexes of a
dg-category $\cat{C}$ organise into a quasicategory $\dgNerve(\cat{C})$ called
the dg-nerve~\cite{a-infinity-dg-nerve, higher-algebra}. When $\cat{C}$ is
pretriangulated, the cycle category $\cat{C}_0$ can also be seen as a
cofibration category~\cite{brown-fibrant-objects, topological-triangulated} which captures the homotopical structure in form of
two distinguished classes of maps called weak equivalences and cofibrations.
This cofibration category $\cat{C}_0$ can in turn be made into a quasicategory
$\fNerve(\cat{C}_0)$ by a construction called the quasicategory of
frames~\cite{szumilo-two-models}, based on resolutions with homotopical Reedy
cofibrant diagrams.

While the transition from a pretriangulated dg-category $\cat{C}$ to its cycle
category $\cat{C}_0$ forgets all higher-degree maps, the shifts and cones
guarantee that no relevant information is lost. This can already be seen via an
abstract argument: The dg-nerve of the dg-category of chain complexes is the
$\infty$-localization of the underlying ordinary category at the weak homotopy
equivalences~\cite[Prop. 1.3.4.5]{higher-algebra}, which generalizes to
arbitrary pretriangulated dg-categories with slight modifications. Since the
quasicategory of frames also implements
this $\infty$-localization~\cite{szumilo-frames}, both of these nerve constructions
translate pretriangulated dg-categories into equivalent quasicategories.

In this paper we construct a simplicial map $B : \dgNerve(\cat{C}) \to
\fNerve(\cat{C}_0)$ that directly exhibits this equivalence for any
pretriangulated dg-category $\cat{C}$. Objects in the resolution diagrams
produced by this map are immediate generalizations of the mapping cylinder of
chain complexes that simultaneously accomodate multiple maps with non-trivial
coherence structure and resolutions of the individual objects. This way we do
not only establish the equivalence of the two models, but provide an
intuitively accessible algebraic motivation for the simplices of the
quasicategory of frames.

After reviewing the theory of cofibration categories in
Section~\ref{sec:cofibration-categories} and pretriangulated dg-categories in
Section~\ref{sec:dg-categories}, we
proceed as follows: In Section~\ref{sec:construction}, we define a Reedy
cofibrant diagram $B(X, f) : D[n] \to \cat{C}_0$ for any $n$-simplex $(X, f)$
of the dg-nerve $\dgNerve(\cat{C})$ of a pretriangulated dg-category $\cat{C}$.
In Section~\ref{sec:homotopical} we see that this diagram is also homotopical and
thus is an $n$-simplex of the quasicategory of frames $\fNerve(\cat{C}_0)$. In
Section~\ref{sec:functor} we show that $B$ is compatible with the simplicial
structure and hence defines a functor of quasicategories. We conclude in
Section~\ref{sec:equivalence} by proving that this functor is an equivalence.

\paragraph{Acknowledgements}

This paper originally was written as a Master's thesis for the MSc in
Mathematics at the University of Bonn under the supervision of Prof.\ Stefan
Schwede. I would also like to thank Dr.\ Fabian Hebestreit and Dr.\ Jamie
Vicary for their careful reading and their valuable suggestions.

\paragraph{Notation}

The category of finite ordinals $\Delta$ consists of objects given by the sets
$[n] = \{ 0, \ldots, n \}$ for all natural numbers $n \in \N$, together with
the order preserving maps as morphisms. Where unambiguous, we also denote a
morphism $f : [n] \to [m]$ as a sequence by $\seq{f(0), \ldots, f(n)}$ and
write $f \cdot g$ for the concatenation of two sequences $f$ and $g$.

We use the theory and notation for quasicategories as discussed
in~\cite{higher-topos-theory}. In particular $\MapR{\cat{C}}$ denotes the space
of right morphisms in a quasicategory $\cat{C}$, and $K \star L$ the join of
two simplicial sets $K$ and $L$.

A graded map $f: X \to Y$ between chain complexes of degree $|f|$ is a family
of maps $f_n : X_n \to Y_{n + |f|}$ of abelian groups for every $n \in \N$.
For graded maps $f : X \to X'$ and $g : Y \to Y'$ between chain
complexes and homogeneous elements $x \in X, y \in Y$, we use the sign
conventions
\begin{align*}
  (f \otimes g)(x \otimes y) &= \sign{|x||g|} f(x) \otimes g(y), \\
  d(x \otimes y) &= dx \otimes y + \sign{|x|} x \otimes dy.
\end{align*}

\section{Cofibration Categories and the Quasicategory of Frames}\label{sec:cofibration-categories}

In some categories the isomorphisms do not capture the appropriate notion of equivalence of the
idealised concepts that the category was intended to model. Prominently, this
is the case for the categories of topological spaces or chain complexes as
models for homotopy types, which are detected by the weak homotopy equivalences
and quasi-isomorphisms, respectively.

Generalising from this, abstract homotopy theory starts with a category in which some
morphisms are marked as weak equivalences and localises the category such that
these morphisms become isomorphisms. It then provides tools to study the
localised category, which often behaves dramatically differently than the
original. Homotopical categories~\cite{dwyer-homotopy} impose some additional
condition on the chosen weak equivalences that makes the theory better behaved:

\begin{definition}
  A \defn{homotopical category} is a category $\cat{C}$ together with a
  collection of morphisms called \defn{weak equivalences} that contains
  all identity morphisms and satisfies the 2-out-of-6 property: If
  \[
    \begin{tikzcd}
      W \ar{r}{f} &
      X \ar{r}{g} &
      Y \ar{r}{h} &
      Z
    \end{tikzcd}
  \]
  is a composable sequence in $\cat{C}$ where $h \circ g$ and
  $g \circ f$ are weak equivalences, then $f$, $g$, $h$, and
  $h \circ g \circ f$ are weak equivalences as well.
  A functor between homotopical categories is homotopical if it
  preserves weak equivalences.
\end{definition}

The data of a homotopical category already suffices to define all usual notions
of abstract homotopy theory, such as homotopy categories, derived functors, and
homotopy limits and colimits~\cite{categorical-homotopy}. In practice, it is
worthwhile to consider homotopical categories with additional structure that
makes calculations much more feasible. Quillen's model
categories~\cite{quillen-homotopical-algebra} equip homotopical categories with
two extra collections of morphisms, called the fibrations and cofibrations,
that interact with each other and the weak equivalences via weak factorisation
systems. Brown's categories of fibrant objects~\cite{brown-fibrant-objects} are
a weakening of this theory whose axioms are both easier to verify and
applicable to more examples, while retaining some of the power of model categories.
In this paper, we will consider the dual of this notion:

\begin{definition}
  A \defn{cofibration category} is a homotopical category $\cat{C}$ equipped
  additionally with a class of morphisms called \defn{cofibrations} (where an
  \defn{acyclic cofibration} is a morphism that is both a cofibration and a
  weak equivalence) such that:
  \begin{enumerate}[noitemsep]
    \item All isomorphisms are acyclic cofibrations. Cofibrations
      are stable under composition. The category $\cat{C}$ has an initial object
      and every morphism from the initial object is a cofibration.
    \item Cofibrations and acyclic cofibrations are closed under pushouts. 
    \item Every morphism in $\cat{C}$ can be factored as the composite of
      a cofibration followed by a weak equivalence.
  \end{enumerate}
\end{definition}

The subcategory of cofibrant objects of a model category canonically becomes a
cofibration category, but conversely it is not always possible or convenient to
find a model structure that gives rise to a particular cofibration category
of interest; see for example~\cite[Section 1.4]{szumilo-two-models}.

Cofibration categories are models of finitely complete homotopy
theories~\cite{szumilo-two-models} and admit a theory of Reedy cofibrant
diagrams, with direct categories playing the role of Reedy categories to
compensate for the lack of fibrations. We replicate a few definitions and
lemmas that are used in the main part of the thesis; for a full discussion
see~\cite{szumilo-two-models} and~\cite{reedy-categories}.

\begin{definition}\leavevmode
  \begin{enumerate}
    \item A category $I$ is \defn{direct} if it admits a functor
      $\deg : I \to \N$ that reflects identities, where $\N$ is equipped
      with the usual order.
  \end{enumerate}
  Let $I$ be a direct category.
  \begin{enumerate}[noitemsep]
    \setcounter{enumi}{1}
    \item Let $i \in I$, then the \defn{latching
        category} $\partial (I \downarrow i)$ is the full subcategory of the
      slice category $I \downarrow i$ without the object $\id_i$.
    \item Let $i \in I$ and $X : I \to \cat{C}$, then the \defn{latching
        object} $L_i X$ is (if it exists) a colimit of
      \[
        \begin{tikzcd}[column sep = 1.25cm]
          \partial (I \downarrow i) \ar{r}{\text{source}} &
          I \ar{r}{X} &
          \cat{C}
        \end{tikzcd}.
      \]
      The \defn{latching map} is the canonical map $L_i X \to X$
      induced by the inclusion
      \[\partial (I \downarrow i) \hookrightarrow (I \downarrow i). \]
    \item Let $\cat{C}$ be a cofibration category.
      A diagram $X : I \to \cat{C}$ is \defn{Reedy cofibrant} if for every $i \in I$
      is the latching object $L_i X$ exists and the latching map is a
      cofibration.
    \item Let $\cat{C}$ be a cofibration category.
      A natural transformation $f : X \to Y$ between Reedy cofibrant diagrams
      $X \to Y$ is a \defn{Reedy cofibration} if for all $i \in I$ the induced map
      \[ X(i) \sqcup_{L_i X} L_i Y \longrightarrow Y(i) \]
      is a cofibration.
    \item A map $I \to J$ of (small) categories is a \defn{sieve} if it is
      injective on objects, fully faithful and closed under precomposition.
  \end{enumerate}
\end{definition}

We will use the following two lemmas that let us manipulate Reedy cofibrant diagrams:

\begin{lem}[{\cite[Theorem 9.3.8(1a)]{reedy-categories}}]
  Let $I$ be a homotopical direct category with finite latching categories and
  $\cat{C}$ a cofibration category. Then the category $\cat{C}^I_\text{R}$ of
  homotopical Reedy cofibrant diagrams $I \to \cat{C}$ and natural
  transformations is a cofibration category with levelwise weak equivalences
  and Reedy cofibrations.
\end{lem}

\begin{lem}[{\cite[Lemma 2.18]{szumilo-frames}}]\label{lem:reedy-modification}
  Let $I \hookrightarrow J$ be a sieve of direct categories and let $X : J \to
  \cat{C}$ be a diagram such that the restriction $X \restriction I$ is Reedy
  cofibrant. Then there exists a Reedy cofibrant diagram $\tilde{X} : J \to
  \cat{C}$ together with a weak equivalence $\tilde{X} \to X$ whose restriction
  to $I$ is the identity map. In particular $\tilde{X} \restriction I = X
  \restriction I$.
\end{lem}

To any cofibration category $\cat{C}$ one can associate a quasicategory by a
construction called the quasicategory of frames $\fNerve(\cat{C})$, built
on homotopical Reedy-cofibrant diagrams of the following shape:

\begin{definition}[\cite{szumilo-two-models}]
  Let $I$ be a homotopical category, then $DI$ is the category with objects
  given by functors $[m] \to I$ and morphisms between
  $\varphi : [m] \to I$ and $\psi : [n] \to I$ given
  by injective maps $\sigma: [m]
  \hookrightarrow [n]$ such that $\varphi = \psi \circ \sigma$.
  The forgetful functor $DI \to I$ sends $x : [n] \to I$ to
  $x(n)$ and creates weak equivalences in $DI$, making $DI$ into a
  homotopical category.
\end{definition}

For any homotopical category $I$, $DI$ is straightforwardly a direct category
by sending $[n] \to I$ to $n$. The latching category $\partial(DI \downarrow
i)$ for some object $i: \lbrack n \rbrack \to I$ consists of restrictions of
$i$ to proper, non-empty subsets of $\lbrack n \rbrack$. Hence $DI$ has finite
latching categories and for any homotopical functor $f : I \to J$ the functor
$Df : DI \to DJ$ induces isomorphisms on latching categories.

\begin{definition}[\cite{szumilo-two-models}]
  Let $\cat{C}$ be a cofibration category. The \defn{quasicategory of frames}
  $\fNerve(\cat{C})$ is the simplicial set where ${\fNerve(\cat{C})}_n$
  is the set of homotopical, Reedy cofibrant diagrams $D[n] \to \cat{C}$ and
  simplicial maps are induced by precomposition.
\end{definition}

\begin{thm}[{\cite[Theorem 3.3]{szumilo-two-models}}]
  Let $\cat{C}$ be a cofibration category. Then $\fNerve(\cat{C})$
  is a finitely cocomplete quasicategory.
\end{thm}

The category $D[0]$ is isomorphic to the category $\Delta_\sharp$ of finite
ordinals and injective order-preserving maps, and every map in $D[0]$ is a weak
equivalence. Hence the $0$-simplices of $\fNerve(\cat{C})$ for some cofibration
category $\cat{C}$ are precisely the homotopy constant Reedy cofibrant
cosemisimplicial resolutions $\Delta_\sharp \to \cat{C}$. This can be seen as a
cosemisimplicial version of a frame in a model category~\cite[Chapter
5]{hovey-model-categories}, which motivates the name.

For a partially ordered set $I$, define $\Sd I$ to be the full subcategory of
$DI$ in which objects are the injective order-preserving maps $[k] \to I$. Then
we have a useful lifting property:

\begin{lem}\label{lem:sd-lift}
  Let $K \subseteq L$ be an inclusion of partially ordered sets, $I
  \hookrightarrow J$ a sieve of direct homotopical categories with finite
  latching categories, and $\cat{C}$ a cofibration category. Then any square
  \[
    \begin{tikzcd}
      I \ar[hookrightarrow]{d} \ar{r}{X} & \cat{C}_\text{R}^{DL} \ar{d}{p} \\
      J \ar{r}{Y} & \cat{C}_\text{R}^{(DK \cup \Sd L)}
    \end{tikzcd}
  \]
  in which $X$ and $Y$ are Reedy cofibrant diagrams admits a Reedy cofibrant
  lift $J \to \cat{C}_\text{R}^{DK}$.
\end{lem}
\begin{proof}
  By~\cite[Lemma 3.19]{szumilo-two-models} the restriction map $p$ is an
  acyclic fibration of cofibration categories, which satisfies the required
  lifting property by~\cite[Lemma 2.15]{szumilo-frames}.
\end{proof}

\begin{lem}\label{lem:exact-product-d}
  Restriction of diagrams in a cofibration category along the canonical map
  $D([k] \times [m]) \to D[k] \times D[m]$ preserves Reedy cofibrant diagrams.
\end{lem}
\begin{proof}
  By~\cite[Lemma 3.11]{szumilo-frames}.
\end{proof}

The quasicategory of frames $\fNerve(\cat{C})$ has more $0$-simplices than the
cofibration category $\cat{C}$ has objects, since homotopy-constant Reedy
cofibrant resolutions are not neccessarily unique. However, any two resolutions
of the same object (or of a pair of weakly equivalent objects) are equivalent
as objects of the quasicategory:

\begin{lem}\label{lem:frames-equivalent}
  Let $\cat{C}$ be a cofibration category, $X, Y : D[0] \to \cat{C}$
  homotopical Reedy cofibrant diagrams such that $X(\seq{0})$ and
  $Y(\seq{0})$ are weakly equivalent. Then $X$ and $Y$ are equivalent
  as $0$-simplices of $\fNerve(\cat{C})$.
\end{lem}
\begin{proof}
  Let $\widehat{[1]}$ be the homotopical category with underlying category $[1]$,
  in which every map is a weak equivalence.  We consider the weak
  equivalence $X(\seq{0}) \to Y(\seq{0})$ as a homotopical functor
  $\widehat{[1]} \to \cat{C}$ and precompose with the projection
  $\Sd\widehat{[1]} \to \widehat{[1]}$ to obtain a homotopical functor
  $\Sd\widehat{[1]} \to \cat{C}$. We then use Lemma~\ref{lem:reedy-modification}
  to modify this to be Reedy cofibrant while keeping the value at
  $\seq{0}$ and $\seq{1}$ the same. The resulting functor
  $\Sd\widehat{[1]} \to \cat{C}$ assembles together with $X$ and $Y$ into
  a homotopical Reedy cofibrant functor
  \[ D(\Delta\{ 0 \} \cup \Delta\{ 1 \}) \cup \Sd \widehat{[1]} \longrightarrow \cat{C}, \]
  which by Lemma~\ref{lem:sd-lift} extends to a homotopical Reedy cofibrant
  functor $D\widehat{[1]} \to \cat{C}$. So by~\cite[Corollary 3.14.]{szumilo-two-models}
  we have an equivalence of $X$ and $Y$ as $0$-simplices.
\end{proof}

\section{Pretriangulated Dg-Categories and the Dg-nerve}\label{sec:dg-categories}

The homotopy category of a stable $(\infty, 1)$-category contains traces of the
stable structure in form of a triangulation. Triangulated categories have
several technical problems, including for example the non-functoriality of
cones, that can be rectified in some cases where the triangulated category
admits a dg-enhancement~\cite{enhanced-triangulated}, i.e.\ if it is the
homology category of a dg-category with certain properties that naturally give
rise to the triangulation.
While many triangulated categories arising from an algebraic context have a
dg-enhancement, including for instance the derived category of any abelian
category~\cite{derived-category-dg-enhancement}, the stable homotopy category
does not~\cite[p. 14]{topological-triangulated}.

\begin{definition}
  A \defn{dg-category} is a category enriched in the category of chain
  complexes, i.e.\ a dg-category $\cat{C}$ consists of:
  \begin{enumerate}[noitemsep]
    \item A collection of objects $\Ob{\cat{C}}$,
    \item a chain complex $\Map{\cat{C}}(X, Y)$ for all $X, Y \in \Ob{\cat{C}}$,
    \item a chain map
      $\circ : \Map{\cat{C}}(Y, Z) \otimes \Map{\cat{C}}(X, Y) \to \Map{\cat{C}}(X, Z)$,
      for all $X, Y, Z \in \Ob{\cat{C}}$, and
    \item an element
      $\id_X \in {\Map{\cat{C}}(X, X)}_0$
     for all $X \in \Ob{\cat{C}}$,
  \end{enumerate}
  such that $\circ$ is associative and the elements $\id_X$ are the identity
  elements of $\circ$. 
  A \defn{dg-functor} $F$ between dg-categories $\cat{C}$ and $\cat{D}$ is
  a functor $\cat{C} \to \cat{D}$ of enriched categories.
\end{definition}

The identity maps are automatically cycles, since composition is a chain map and thus
\[ d\id_X = d(\id_X \circ \id_X) = d\id_X \circ \id_X + \id_X \circ d \id_X = 2d\id_X. \]

For more details on enriched category theory, see for instance~\cite[Chapter
3]{categorical-homotopy}. Dg-categories can also be defined more generally for
chain complexes of $A$-modules for any commutative ring $A$. All constructions
in this paper are independent of the ground ring, so we simplify by only
considering chain complexes over abelian groups. We will further assume all
dg-categories to be small.

The category of (small) chain complexes $\Ch$ is closed symmetric monoidal with the
tensor product as the monoidal product. Hence it can be enriched over itself,
with mapping chain complexes given by \[
  {\Map{\ChCh}(X, Y)}_n = \prod_{k \in \mathbb{Z}} \Ab(X_k, Y_{k + n}) \] and
the differential acting on $f \in {\Map{\ChCh}(X, Y)}_n$ by \[ df = d_Y \circ f
  - {(-1)}^n f \circ d_X. \]
A dg-functor $\cat{C} \to \ChCh$ will also be called a $\cat{C}$-module.
Such dg-functors take the place of presheaves in ordinary category theory
and satisfy an enriched version of the Yoneda lemma.

\begin{definition}
  Let $\cat{C}$ be a dg-category. The cycle category
  $\cat{C}_0$ of $\cat{C}$ is the category with the same objects as
  $\cat{C}$ and morphisms given by $0$-cycles of the mapping complexes:
  \[ \Hom{\cat{C}_0}(X, Y) = \{ f \in {\Map{\cat{C}}(X, Y)}_0 \mid df = 0 \}. \]
  The homology category
  $H(\cat{C})$ of $\cat{C}$ is the category with the same objects as $\cat{C}$ and
  morphisms given by the homology of the mapping complexes in degree $0$:
  \[ \Hom{H(\cat{C})}(X, Y) = H_0(\Map{\cat{C}}(X, Y)). \]
\end{definition}

We observe that $\Z \langle 0 \rangle$, the chain complex that is $\Z$ in
degree $0$ and vanishes anywhere else, is the tensor unit in the monoidal
structure of chain complexes, and chain maps $\Z \langle 0 \rangle \to C$ are
precisely the $0$-cycles of $C$, the cycle category is the same as the
underlying ordinary category in the sense of enriched category theory.

\begin{definition}
  Let $\cat{C}$ be a dg-category. A map $f : X \to Y$ in the cycle category
  $\cat{C}_0$ is a weak equivalence, if it becomes an isomorphism in the
  homology category.
\end{definition}

By unrolling the definitions, we can see that a $0$-cycle $f : X \to Y$ is a
weak equivalence if there exists a $0$-cycle $g : Y \to X$ and $1$-chains $h
\in {\Map{\cat{C}}(X, X)}_1, h' \in {\Map{\cat{C}}(Y, Y)}_1$ such that \[ f
  \circ g - \id_Y = dh', \quad g \circ f - \id_X = dh. \] In particular, weak
equivalences in $\ChCh$ are precisely the chain homotopy equivalences.

\begin{definition}
  Let $\cat{C}$ be a dg-category.
  The \defn{$n$-translation} of an object $X \in \cat{C}$ for some $n \in
  \mathbb{Z}$ is an object $X[n] \in \cat{C}$ representing the
  $n$-translated mapping complex: \[ \Map{\cat{C}}(X[n], -) \cong \Map{\cat{C}}(X, -)[-n]. \]
  The \defn{mapping cone} of a $0$-cycle $f : X \to Y$ is an object
  $\Cone(f)$ representing the mapping cone of the induced map on mapping complexes:
  \[ 
    \Map{\cat{C}}(\Cone(f), -) \cong
    \Cone(f^* : \Map{\cat{C}}(X, -) \to \Map{\cat{C}}(Y, -)).
  \]
  A dg-category $\cat{C}$ is \defn{pretriangulated} if it has a zero object, it
  admits translations of all objects, and it admits mapping cones for all
  $0$-cycles.
\end{definition}

When $\cat{C}$ is a pretriangulated dg-category, its homology category
$H(\cat{C})$ can be canonically triangulated: the shift functor is induced
by the translations, and the distinguished triangles come from the image of
mapping cone sequences.
More importantly for this paper, the cycle category of a pretriangulated
dg-category can be made into a cofibration category:

\begin{definition}
  Let $\cat{C}$ be a pretriangulated dg-category. A $0$-cycle $i : X \to Y$
  is a \defn{cofibration} if the precomposition map
  $i^* : \Map{\cat{C}}(Y, -) \longrightarrow \Map{\cat{C}}(X, -)$
  is surjective as a map of $\cat{C}$-modules and has a representable kernel.
\end{definition}

\begin{prop}[{\cite[Prop 3.2]{topological-triangulated}}]
  Let $\cat{C}$ be a pretriangulated dg-category. Then the cofibrations
  and weak equivalences make the cycle category $\cat{C}_0$ into a 
  cofibration category.
\end{prop}

The following lemma allows for particular convenient inductive constructions of
objects in pretriangulated dg-categories via extensions of representable
$\cat{C}$-modules:

\begin{lem}\label{lem:dg-ses-representable}
  Let $\cat{C}$ be a pretriangulated dg-category and
  \[
    \begin{tikzcd}
      0 \ar{r} & A \ar{r} & B \ar{r} & C \ar{r} & 0
    \end{tikzcd}
  \]
  a short exact sequence of $\cat{C}$-modules such that $A$ and $C$
  are representable. Then $B$ is representable as well. Moreover,
  the map of $\cat{C}_0$ represented by $B \to C$ is a cofibration.
\end{lem}
\begin{proof}
  Representability of $B$ is part of the proof of~\cite[Prop.
  3.2]{topological-triangulated}. Since $B \to C$ is surjective
  and its kernel $A$ is representable, the map
  represented by $B \to C$ is a cofibration.
\end{proof}

\begin{lem}
  \label{lem:acyclic-cofibration-retraction}
  Let $\cat{C}$ be a pretriangulated dg-category and $i : A \to B$ an
  acyclic cofibration. Then there exists a retraction $p : B \to A$ of
  $i$ and a homotopy $h : i \circ p \simeq \id$ with $h \circ i = 0$.
\end{lem}
\begin{proof}
  The retraction exists by~\cite[Prop 3.2]{topological-triangulated}
  since every object is fibrant. Then 
  \[ (i \circ p - \id) \circ i = i \circ p \circ i - i = 0, \]
  so $(i \circ p - \id)$ is a cycle in the kernel of the precomposition map $i^*$. Since $i$ is a cofibration, $\ker(i^*)$ is contractible and thus
  there exists a $h \in {\ker(i^*)}_1$ such that
  \[ dh = i \circ p - \id. \qedhere\]
\end{proof}

A dg-category can be transformed into a quasicategory by first considering the
mapping complexes as simplicial sets via the Dold-Kan correspondence to obtain
a simplicially enriched category and then taking the homotopy coherent nerve.
Alternatively, there is a more direct construction:

\begin{definition}[\cite{a-infinity-dg-nerve}]\label{def:dg-nerve}
  Let $\cat{C}$ be a dg-category. The \defn{dg-nerve} $\dgNerve(\cat{C})$ is
  the simplicial set where an $n$-simplex
  consists of a family of objects $X_0, \ldots, X_n$
  together with a family of maps:

  \begin{enumerate}
    \item For each sequence $\seq{i_0, \ldots, i_k}: [k] \to [n]$ with $k > 0$ there is a chosen map
      \[f(\seq{i_0, \ldots, i_k}) \in {\Map{\cat{C}}(X_{i_0}, X_{i_k})}_{k - 1}\]
    \item The maps satisfy a differential coherence condition:
      \begin{align*}
        -df(\seq{i_0, \ldots, i_k}) &=
        \sum_{j = 1}^{k - 1}
        \sign{j}
        f(\seq{i_0, \ldots, \hat{i_j}, \ldots, i_k})  \\
        &+\sum_{j = 1}^{k - 1}
        \sign{(j - 1)k}
        f(\seq{i_j, \ldots, i_k})
        \circ
        f(\seq{i_0, \ldots, i_j}),
      \end{align*}
    \item The maps satisfy strict unitality:
      \begin{alignat*}{3}
        f(\seq{i_0, i_1}) &= \id_{X_{i_0}} &&\qquad \text{(if $i_0 = i_1$)} \\
        f(\seq{i_0, \ldots, i_k}) &= 0 &&\qquad \text{(if $i_j = i_{j + 1}$ for some $j$)}.
      \end{alignat*}
  \end{enumerate}
  An order-preserving map $\alpha : [m] \to [n]$ acts on such an $m$-simplex
  by reindexing the objects and by precomposition of the sequences.
\end{definition}

Simplices of the dg-nerve of a dg-category $\cat{C}$ can be seen as
homotopy-coherent diagrams, where the coherence data is given by higher-degree
morphisms. Concretely, the $0$-simplices of $\dgNerve(\cat{C})$ are just the
objects of $\cat{C}$, the $1$-simplices are the maps of the cycle category of
$\cat{C}$, and a $2$-simplex is given by a triple of objects $X_0, X_1, X_2$ and
$0$-cycles
\begin{align*}
  f(\seq{01}) &\in {\Map{\cat{C}}(X_0, X_1)}_0, &
  f(\seq{12}) &\in {\Map{\cat{C}}(X_1, X_2)}_0, &
  f(\seq{02}) &\in {\Map{\cat{C}}(X_0, X_2)}_0,
\end{align*}
together with an element $f(\seq{012}) \in {\Map{\cat{C}}(X_0, X_2)}_1$
such that \[df(\seq{012}) = f(\seq{02}) - f(\seq{12}) \circ f(\seq{01}).\]

There is a map $\Nerve(\cat{C}_0) \to \dgNerve(\cat{C})$ that sends every
$\sigma : [n] \to \cat{C}_0$ to the $n$-simplex given by the family of objects
$\sigma(0), \ldots, \sigma(n)$ and the family of maps with $f(\seq{i_0, i_1}) =
\sigma(i_0 \to i_1)$ and $f(\seq{i_0, \ldots, i_k}) = 0$ for any $k \neq 1$.

\begin{prop}
  Let $\cat{C}$ be a dg-category, then $\dgNerve(\cat{C})$ is a quasicategory
  and the map $\Nerve(\cat{C}_0) \to \dgNerve(\cat{C})$ descends to an isomorphism
  $H(\cat{C}) \to \Ho(\dgNerve(\cat{C}))$ from the homology category to the homotopy
  category.
\end{prop}
\begin{proof}
  By~\cite[Proposition 2.2.12]{a-infinity-dg-nerve} and~\cite[Remark 1.3.1.11]{higher-algebra}.
\end{proof}

\begin{remark}\label{remark:a-infinity-functor}
  Consider the poset $[n]$ as a dg-category with mapping complexes either
  $\Z\langle 0 \rangle$ or $0$, then the dg-nerve of a dg-category $\cat{C}$ as
  defined above is precisely the simplicial set of strictly unital
  $\Ainfty$-functors $[n] \to \cat{C}$, with the simplicial maps given by
  precomposition~\cite{a-infinity-dg-nerve}. 
  By strict unitality, a simplex of the dg-nerve is already determined by the
  maps associated to injective sequences, but including the non-injective sequences
  in the definition makes some constructions (in particular that of the
  simplicial action) more convenient. This definition is still equivalent to
  that in~\cite{higher-algebra}, which differs only in the omission of the maps
  forced by strict unitality and in the choice of signs.
\end{remark}

We conclude the exposition by a technically useful alternative view on
simplices of the dg-nerve that simplifies the calculation in the main part of
the thesis. 

\begin{definition}
  For $n \geq 0$, define $\PathAlg(n)$ to be the dg-coalgebra that as a graded
  abelian group is given in degree $k > 0$ by the freely generated abelian
  group
  \[ {\PathAlg(n)}_k = \FreeAb{\{ \alpha : [k] \to [n] \}} \]
  and which vanishes in degrees $k \leq 0$.
  The differential of $\PathAlg(n)$ sends a homogeneous element $\alpha : [k] \to [n]$ to the
  alternating sum of its inner faces:
  \[ 
    d\seq{\alpha_0, \ldots, \alpha_k} = 
    \sum_{j = 1}^{k - 1} \sign{j}
    \seq{\alpha_0, \ldots, \hat{\alpha_j}, \ldots, \alpha_k}.
  \]
  The comultiplication
  $\Delta : \PathAlg(n) \to \PathAlg(n) \otimes \PathAlg(n)$
  splits sequences:
  \[
    \Delta \seq{\alpha_0, \ldots, \alpha_k} =
    \sum_{j = 1}^{k - 1}
    \sign{j(k - j)}
    \seq{\alpha_j, \ldots, \alpha_k}
    \otimes
    \seq{\alpha_0, \ldots, \alpha_j}.
  \]
\end{definition}

A dg-category $\cat{C}$ induces a dg-algebra
\[
\dgCatAlg{\cat{C}} := \bigoplus_{X, Y} \Map{\cat{C}}(X, Y)
\]
with the multiplication $m$ that composes compatible maps and is $0$ otherwise. 
The chain complex of maps $\PathAlg(n) \to \dgCatAlg{\cat{C}}$ then is a dg-algebra with
the convolution product:
\[ p * q := m \circ (p \otimes q) \circ \Delta. \]
This leads to the following characterisation of simplices of the dg-nerve:

\begin{lem}\label{lem:nerve-maurer-cartan}
  Let $\cat{C}$ be a dg-category. An $n$-simplex of $\dgNerve(\cat{C})$
  corresponds to a family of objects $X_0, \ldots, X_n$ and a graded map $f :
  \PathAlg(n) \to \dgCatAlg{\cat{C}}$ of degree $-1$ satisfying the following conditions:
  \begin{enumerate}
    \item Sequences are sent to the mapping spaces between the correct objects: 
      \[f(\seq{i_0, \ldots, i_k}) \in {\Map{\cat{C}}(X_{i_0}, X_{i_k})}_{k - 1}\]
    \item $f$ satisfies the Maurer-Cartan condition in the convolution algebra:
      \[df + f * f = 0.\]
    \item $f$ satisfies strict unitality:
      \begin{alignat*}{3}
        f(\seq{i_0, i_1}) &= \id_{X_{i_0}} &&\qquad \text{(if $i_0 = i_1$)} \\
        f(\seq{i_0, \ldots, i_k}) &= 0 &&\qquad \text{(if $i_j = i_{j + 1}$ for some $j$)}.
      \end{alignat*}
  \end{enumerate}
\end{lem}

\begin{proof}
  Conditions 1 and 3 correspond directly to those in Definition~\ref{def:dg-nerve}.
  The Maurer-Cartan condition is equivalent to the differential coherence condition:
  Let $\alpha : [k] \to [n]$ be an element of $\PathAlg(n)$, then we have:
  \begin{align*}
    f(d\alpha) &= 
    \sum_{j = 1}^{k - 1}
    \sign{j}
    f(\seq{\alpha_0, \ldots, \hat{\alpha_j}, \ldots, \alpha_k})\\
    (f * f)(\alpha) &=
    \sum_{j = 1}^{k - 1}
    \sign{(j - 1)k}
    f(\seq{\alpha_j, \ldots, \alpha_k})
    \circ
    f(\seq{\alpha_0, \ldots, \alpha_j}).\qedhere
  \end{align*}
\end{proof}

\begin{remark}
  Graded maps of degree $-1$ from a coalgebra into an algebra satisfying the Maurer-Cartan
  condition are also called twisted cochains. This fits into a more general viewpoint:
  When we consider an $n$-simplex of $\dgNerve(\cat{C})$ as an an $\Ainfty$-functor
  $[n] \to \cat{C}$ like in Remark~\ref{remark:a-infinity-functor}, this
  induces a map $\dgCatAlg{[n]} \to \dgCatAlg{\cat{C}}$ of $\Ainfty$-algebras,
  which corresponds in the terminology of~\cite{keller-a-infinity} to a twisting cochain from the
  bar construction $B\dgCatAlg{[n]}$ to $\dgCatAlg{\cat{C}}$. We can unroll the
  definitions to see that $B\dgCatAlg{[n]}$ is $\PathAlg(n)$.
\end{remark}

\section{Constructing Resolutions}\label{sec:construction}

The dg-nerve $\dgNerve(\cat{C})$ of a dg-category $\cat{C}$ makes use of the
higher-degree functions in the mapping complexes of $\cat{C}$, whereas the
transition from $\cat{C}$ to its underlying ordinary category $\cat{C}_0$
forgets all but the 0-cycles, with only some hints about higher dimensions
encoded in weak equivalences and cofibrations. When $\cat{C}$ is
pretriangulated, however, all of the homotopical structure of $\cat{C}$ can be
recovered, which we show by an explicit construction that coherently encodes an
$n$-simplex of the dg-nerve into a homotopical Reedy cofibrant diagram $D[n]
\to \cat{C}_0$.

The core idea is to encode the maps of an simplex in the dg-nerve as the
differential of a complex, and then filter this complex to obtain a resolution
by inclusion maps. The most prominent example of such a construction is the
mapping cylinder of a chain map. To generalize from a cylinder to a shape that
can accomodate higher coherence structure, we define:

\begin{definition}
  For an object $\alpha \in D[n]$ define $\PathMod(\alpha)$ to be the
  $\PathAlg(n)$-comodule that in
  degree $k \geq 0$ is the free abelian group
  \[ {\PathMod(\alpha)}_k = \FreeAb{\{ (i : \beta \hookrightarrow \alpha) \in (D[n] \downarrow \alpha)
    \mid \beta : [k] \to [n] \}} \]
  together with the differential 
  \[ 
    d^*\seq{i_0, \ldots, i_k} = 
    \sum_{j = 1}^{k} \sign{j}
    \seq{i_0, \ldots, \hat{i_j}, \ldots, i_k},
  \]
  and the comultiplication
  $\delta : \PathMod(\alpha) \to \PathMod(\alpha) \otimes \PathAlg(n)$
  \[
    \delta \seq{i_0, \ldots, i_k} =
    \sum_{j = 1}^{k}
    \sign{j(k - j)}
    \seq{i_j, \ldots, i_k}
    \otimes
    (\alpha \circ \seq{i_0, \ldots, i_j}).
  \]
  By postcomposition this definition extends to a functor from
  $D[n]$ to $\PathAlg(n)$-comodules.
\end{definition}

Consider a dg-category $\cat{C}$ as a dg-algebra with multiplication $m$ that
composes compatible maps and let $\alpha \in D[n]$. The chain complex of
graded maps $\varphi: \PathMod(\alpha) \to \dgCatAlg{\cat{C}}$ becomes a dg-module over
the convolution algebra of maps $f : \PathAlg(n) \to \dgCatAlg{\cat{C}}$ via the multiplication
\[ \varphi * f := m \circ (\varphi \otimes f) \circ \delta. \]

In the category of chain complexes, we can directly construct an object with a
specified differential that generalizes that of the mapping cylinder by
applying the higher coherence maps. In an arbirary pretriangulated dg-category
$\cat{C}$, we have to proceed indirectly. We specify our main construction as a
diagram of $\cat{C}$-modules, which we show to be representable shortly after:

\begin{definition}
  Let $\cat{C}$ be a dg-category, $(X, f)$ an $n$-simplex of
  $\dgNerve(\cat{C})$ and $\alpha \in D[n]$. Then
  let $C(X, f)(\alpha)$ denote the $\cat{C}$-module defined by the graded abelian
  group
  \[ 
    C(X, f)(\alpha) =
    \prod_{i \in D[n] \downarrow \alpha} \Map{\cat{C}}(X_{(\alpha \circ i)(0)}, -) \subseteq
    \Map{\Ch}(\PathMod(\alpha), \oplus_{X \in \cat{C}} \Map{\cat{C}}(X, -))
  \]
  together with the differential that sends homogeneous elements $\varphi \in C(X, f)(\alpha)$ to
  \begin{equation}\label{eqn:c-differential}
    d_f \varphi := d\varphi - \sign{|\varphi|} \varphi * f,
  \end{equation}
  where $d$ is the differential of the underlying mapping space.
  By functoriality of $\PathMod$, this definition extends to a functor
  from $\op{D[n]}$ to $\cat{C}$-modules.
\end{definition}

Concretely, for any object $Y$ a homogeneous element $\varphi \in C(X, f)(\alpha)(Y)$ of degree $|\varphi|$
assigns to every $i \in D[n] \downarrow \alpha$ of the form
\[
  \begin{tikzcd}
    \lbrack b \rbrack \ar[hookrightarrow]{rr}{i} \ar[swap]{dr}{\beta} & & \lbrack a \rbrack \ar{dl}{\alpha} \\
    & \lbrack n \rbrack
  \end{tikzcd}
\]
a homogeneous map $X_{(\alpha \circ i)(0)} \to Y$ of degree $b + |\varphi|$.
We can expand the differential (\ref{eqn:c-differential})
on an element $\seq{i_0, \ldots, i_k} \in \PathMod(\alpha)$ as:
\begin{align*}
  d_f \varphi(\seq{i_0, \ldots, i_k}) &=
  \sum_{j = 1}^k \sign{j} \seq{i_0, \ldots, \hat{i_j}, \ldots, i_k} \\
  &- \sum_{j = 1}^k \sign{|\varphi| + k(j - 1)}
  \varphi(\seq{i_j, \ldots, i_k}) \circ f(\seq{\alpha(i_0), \ldots, \alpha(i_j)}) \\
  &- \sign{|\varphi|} d(\varphi(\seq{i_0, \ldots, i_k})).
\end{align*}

In Lemma\ \ref{lem:nerve-maurer-cartan} we have seen that the maps $f :
\PathAlg(n) \to \Map{\cat{C}}(X, X)$ arising from an $n$-simplex of the dg-nerve
$\dgNerve(\cat{C})$ satisfy the Maurer-Cartan condition $df + f * f = 0$ in the
convolution algebra. Here we can use this fact to show:

\begin{lem}
  Let $\cat{C}$ be a dg-category, $(X, f)$ an $n$-simplex of
  $\dgNerve(\cat{C})$ and $\alpha \in D[n]$. Then $d_f^2 = 0$,
  so $C(X, f)(\alpha)$ is a well-defined $\cat{C}$-module.
\end{lem}
\begin{proof}
  Let $\varphi \in C(X, f)(\alpha)$ be homogeneous. Then we calculate:
  \begin{align*}
    d_f^2 \varphi &= d_f (d\varphi - \sign{|\varphi|} \varphi * f) \\
    &= d^2 \varphi - \sign{|\varphi|} d(\varphi * f)
    + \sign{|\varphi|} d\varphi * f - (\varphi * f) * f \\
    &= -\sign{|\varphi|} d\varphi * f - \varphi * df 
    + \sign{|\varphi|} d\varphi * f - \varphi * (f * f) \\
    &= - \varphi * (df + f * f).
  \end{align*}
  By Lemma\ \ref{lem:nerve-maurer-cartan} we have that $df + f * f = 0$,
  so this last term vanishes.
\end{proof}

Because the generating set of $\PathMod(\alpha)$ is finite, we can use an
inductive argument to construct representations of $C(X, f)$ in any
pretriangulated dg-category:

\begin{prop}\label{prop:b-representable-reedy}
  Let $\cat{C}$ be a pretriangulated dg-category, $(X, f)$ an $n$-simplex
  of $\dgNerve(\cat{C})$ and $\alpha \in D[n]$. Then the $\cat{C}$-module
  $C(X, f)(\alpha)$ is representable by some object $B(X, f)(\alpha)$. By
  varying $\alpha$, this assembles into a Reedy cofibrant functor 
  \[ B(X, f) : D[n] \longrightarrow \cat{C}. \]
\end{prop}
\begin{proof}
  For a sieve $S \subseteq (D[n] \downarrow \alpha)$, denote by $C(X, f)(S)$ the
  $\cat{C}$-module
  \[ C(X, f)(S) := \prod_{\beta \in S} \Map{\cat{C}}(X_{(\alpha \circ \beta)(0)}, -), \]
  obtained from $C(X, f)$ by restriction of the product to indices in $S$. This
  can equivalently be seen as the quotient of $C(X, f)$ by those components of the
  product not contained in $S$.

  For the empty sieve $\emptyset$ we have $C(X, f)(\emptyset) = 0$, which is
  representable by the zero object of $\cat{C}$. Suppose that $C(X, f)(S)$
  is representable for some sieve $S$ and let
  $i \in (D[n] \downarrow \alpha) \setminus S$ be minimal, where $i : \beta \hookrightarrow \alpha$ for some $\beta : [k] \to [n]$.
  By restriction of the product we then obtain a short exact sequence of
  $\cat{C}$-modules
  \begin{equation}
    \label{eqn:ses-resolution-sieve}
    \begin{tikzcd}[column sep = 0.5cm]
      0 \ar{r} &
      \Map{\cat{C}}(X_{\beta(0)}, -)[-k] \ar{r} &
      C(X, f)(S \cup \{ i \}) \ar{r} &
      C(X, f)(S) \ar{r} &
      0
    \end{tikzcd}
  \end{equation}

  $\Map{\cat{C}}(X_{\beta(0)}, -)[-k]$ is representable by the shifted
  object $X[k]$ and $C(X, f)(S)$ is representable by the inductive assumption.
  Thus by Lemma\ \ref{lem:dg-ses-representable} we have that
  $S(X, f)(S \cup \{\beta\})$
  is representable as well.
  For Reedy cofibrancy, observe that the latching map for an object $\alpha \in
  D[n]$ is represented by the map of $\cat{C}$-modules
  \[ C(X, f)(D[n] \downarrow \alpha) \longrightarrow C(X, f)(\partial(D[n] \downarrow \alpha)), \]
  which is a cofibration by construction.
\end{proof}

\begin{example}
  Let $\cat{C}$ be a pretriangulated dg-category and 
  $(X, f)$ a $0$-simplex in the dg-nerve of $\cat{C}$.
  Then $f$ vanishes anywhere but for $f(\seq{0, 0}) = \id$,
  so for any $\alpha : [n] \to [0]$ we have a natural isomorphism
  \[ 
    C(X, f)(\alpha) \cong
    \Map{\Ch}(\Moore_*(\Delta[n]), \Map{\cat{C}}(X_0, -))
  \]
  where $\Moore_*(\Delta[n])$ is the complex of normalized chains of
  $\Delta[n]$. Thus we see that the representation $B(X, f)$ is given by a
  copowering with $\Moore_*(\Delta[n])$.
\end{example}

\begin{example}
  Let $\cat{C}$ be a pretriangulated dg-category and
  $(X, f)$ an $n$-simplex of $\dgNerve(\cat{C})$. For any $i \in [n]$
  we have
  \[ 
    C(X, f)(\seq{i}) = 
    \Map{\Ch}(\Z \langle 0 \rangle, \Map{\cat{C}}(X_i, -)) \cong
    \Map{\cat{C}}(X_i, -),
  \]
  and thus $B(X, f)(\seq{i}) \cong X_i$.
\end{example}

\begin{example}
  Let $(X, f)$ be an $n$-simplex of $\dgNerve(\ChCh)$. Then
  for any $0 \leq i < j \leq n$, $B(X, f)(\seq{i, j})$ is the mapping
  cylinder of the chain map $f(\seq{i, j}) : X_i \to X_j$. The
  inclusions of the sequences $\seq{i}$ and $\seq{j}$ into $\seq{i, j}$ induce
  the standard inclusions of $X_i$ and $X_j$ into the mapping cylinder.
\end{example}

\section{The Resolutions are Homotopical}\label{sec:homotopical}

The weak equivalences in the category $D[n]$ are those injections of sequences
that preserve the maximum element. For a pretriangulated dg-category $\cat{C}$
and an $n$-simplex $(X, f)$, we will show that the functor $B(X, f) : D[n] \to
\cat{C}$ is homotopical.
We begin by defining a homotopy retraction of $C(X, f)(\alpha)$ onto its
restriction to the last index of $\alpha$. For notational convenience, we
identify the isomorphic $\cat{C}$-modules $C(X, f)(\seq{\max \alpha})$ with
$\Map{\cat{C}}(X_{\max \alpha}, -)$.

\begin{lem}\label{lem:last-inclusion-equivalence}
  Let $\cat{C}$ be a dg-category, $(X, f)$ an $n$-simplex of the dg-nerve
  of $\cat{C}$, and $\alpha : [k] \to [n]$. Consider the map of
  $\cat{C}$-modules
  \[ I : C(X, f)(\alpha) \longrightarrow C(X, f)(\seq{\max \alpha}) \]
  induced by the inclusion $\seq{\max \alpha} \hookrightarrow \alpha$
  at the last index. Then
  \begin{align*}
    P : \Map{\cat{C}}(X_{\max \alpha}, -) &\longrightarrow C(X, f)(\alpha) \\
    \omega \longmapsto i &\longmapsto \sign{|i|} \omega \circ f((\alpha \circ i) \cdot \seq{\max \alpha})
  \end{align*}
  is a closed $0$-cycle and a retraction of $I$. Moreover, 
  \begin{alignat*}{3}
    H : C(X, f)(\alpha) &\longrightarrow C(X, f)(\alpha) \\
    \varphi \longmapsto i &\longmapsto \sign{|\varphi| + |i|} \varphi(i \cdot \seq{k})
    &&\quad \text{(if $\max i < k$)} \\
    \varphi \longmapsto i &\longmapsto 0 &&\quad \text{(if $\max i = k$)}.
  \end{alignat*}
  defines a homotopy $P \circ I \simeq \id$.
\end{lem}
\begin{proof}
  Let $\omega \in \Map{\cat{C}}(X_{\max \alpha}, A)$ for any object $A$ and
  $i : \beta \hookrightarrow \alpha$. Then we have
  \begin{align}
    (dP)(\omega)(i)
    &= \sign{|i|} d\omega \circ f(\alpha \circ i)
    - \sign{|i|} d\omega \circ f(\alpha \circ i)
    \label{eqn:homotopy-projection-1} \\
    &+ \sign{|\omega| + |i|} \omega \circ df(\alpha \circ i \cdot \seq{\max \alpha})
    - \sign{|\omega|} (P(\omega) * f)(i)
    \label{eqn:homotopy-projection-2} \\
    &- \sign{|\omega| + |i|} \omega \circ f(d(\alpha \circ i \cdot \seq{\max \alpha}))
    - \sign{|\omega|} P(\omega)(d^*i).
    \label{eqn:homotopy-projection-3}
  \end{align}

  The term (\ref{eqn:homotopy-projection-1}) vanishes immediately.
  In (\ref{eqn:homotopy-projection-2}) we can substitute $-(f * f)$ for $df$ and
  observe that the convolutions cancel after expansion. Analogously, the
  alternating face sums in (\ref{eqn:homotopy-projection-3}) cancel after
  expansion. Thus $P$ is a chain map.

  For any $\omega \in \Map{\cat{C}}(X_{\max \alpha} -)$, we see that
  \[ 
    (I \circ P)(\omega) =
    P(\omega)(\seq{\max \alpha}) =
    \omega \circ f(\seq{\max \alpha, \max \alpha}) =
    \omega,
  \]
  so $P$ is a retraction of $I$.
  
  It remains to show that $dH = P \circ I - \id$. For $\varphi \in C(X, f)(\alpha)$, we have
  \begin{align}
    (dH)(\varphi)
    &= d \circ H(\varphi) + H(d \circ \varphi)
    \label{eqn:homotopy-1} \\
    &+ \sign{|\varphi|} H(\varphi) * f - \sign{|\varphi|} H(\varphi * f)
    \label{eqn:homotopy-2} \\
    &+ \sign{|\varphi|} H(\varphi) \circ d^* - \sign{|\varphi|} H(\varphi \circ d^*).
    \label{eqn:homotopy-3}
  \end{align}
  We evaluate this on elements $i \in \PathMod(\alpha)$ in three cases to see
  that $dH = P \circ I - \id{}$:

  \begin{enumerate}
    \item Let $i \in \PathMod(\alpha)$ with $\max i < k$. The terms in
      (\ref{eqn:homotopy-1}) cancel each other. Evaluated on $i$, the
      convolutions in (\ref{eqn:homotopy-2}) cancel except for one summand
      \[ 
      \sign{|i|} \varphi(\seq{k}) \circ f(\alpha \circ i \cdot \seq{\max \alpha})
      = (P \circ I)(\varphi)(i).
      \]
      Similarly, the alternating face sums of (\ref{eqn:homotopy-3}) cancel except for one summand
      \[- \varphi(i) = -\id{}(\varphi)(i).\]

    \item Now consider $i \in \PathMod(\alpha)$ with $\max i = |\alpha|$ and
      $|i| > 0$ and calculate
      \begin{align}
        (dH)(\varphi)(i)
        &= \sign{|\varphi|} (H(\varphi) * f)(i) \label{eqn:homotopy-last-1} \\
        &+ \sign{|\varphi|} (H(\varphi) \circ d^*)(i) \label{eqn:homotopy-last-2}.
      \end{align}
      Then (\ref{eqn:homotopy-last-1}) vanishes since $H$ vanishes on any suffix of
      $i$, so it agrees with
      \[ (P \circ I)(\varphi)(i) = \sign{|i|} \varphi(\seq{k})
        \circ f(\alpha \circ i \cdot \seq{\max \alpha}) = 0, \]
      which is zero since $\alpha \circ i \cdot \seq{\max \alpha}$ is not injective
      and has more than two elements. The only summand in $d^*i$
      that does not contain $k$ is the prefix of $i$ of length $|i| - 1$, which
      is promptly extended again in (\ref{eqn:homotopy-last-2}) to
      the original $i$. So (\ref{eqn:homotopy-last-2}) evaluates to
      $\varphi(i)$.

    \item Finally observe that $dH(\varphi)(\seq{k}) = 0$ as well as
      \[
         (P \circ I - \id{})(\varphi)(\seq{k}) = 
         \varphi(\seq{k}) \circ f(\seq{\max \alpha, \max \alpha}) - \varphi(\seq{k}) = 0.\qedhere
       \]
  \end{enumerate}
\end{proof}

The case distinction in the definition of the homotopy $H$ was neccessary since
$\PathMod(\alpha)$ is not closed under appending the maximum element of
$\alpha$.

\begin{prop}\label{prop:b-homotopical}
  Let $\cat{C}$ be a dg-category and $(X, f)$ an $n$-simplex of the
  dg-nerve of $\cat{C}$.
  \begin{enumerate}[noitemsep]
    \item The functor $C(X, f): \op{D[n]} \to \Mod{\cat{C}}$ is homotopical.
    \item The functor $B(X, f): D[n] \to {\cat{C}}_0$ is homotopical.
  \end{enumerate}
\end{prop}
\begin{proof}
  For the first claim, we consider order-preserving maps $\alpha : [k] \to [n], \beta : [r] \to [n]$
  and let $i : \alpha \hookrightarrow \beta$ be a weak equivalence
  in $D[n]$. Then we have a diagram
  \[
    \begin{tikzcd}
      \alpha \ar{rr}{i} && \beta \\
      & \seq{\alpha(k)} \ar[swap]{ur}{q} \ar{ul}{p}
    \end{tikzcd}
  \]
  in $D[n]$ where $p$ is the inclusion at the last index and
  $q$ the inclusion at index $i(k)$.
  Then we get an induced diagram of $\cat{C}$-modules
  \[
    \begin{tikzcd}
      C(X, f)(\alpha) \ar[swap]{dr}{p^*} &&
      C(X, f)(\beta) \ar{dl}{q^*} \ar[swap]{ll}{i^*} \\
      & \Map{\cat{C}}(X_{\alpha(k)}, -)
    \end{tikzcd}
  \]
  $p^*$ is a weak equivalence by Lemma\ \ref{lem:last-inclusion-equivalence}.
  If $i(k) = r$, then $q$ also is the inclusion at the last index and thus
  also a weak equivalence, whence $i^*$ is a weak equivalence by 2-out-of-3.
  Otherwise $i(k) < r$, but then
  \begin{align*}
    C(X, f)(\beta) &\longrightarrow \Map{\cat{C}}(X_{\alpha(k)}, -) \\
    \varphi &\longmapsto \sign{|\varphi|} \varphi(\seq{i(k), r})
  \end{align*}
  is a homotopy between $q^*$ and the map induced by the inclusion at $r$,
  so $q^*$ is a weak equivalence as well.
  
  Now the second claim follows since a chain map of representable
  $\cat{C}$-modules is a chain homotopy equivalence precisely if it
  represents a map in $\cat{C}_0$ that is a weak equivalence.
\end{proof}

\section{Functor of Quasicategories}\label{sec:functor}

For any $n$-simplex $(X, f)$ of the dg-nerve $\dgNerve(\cat{C})$ of a
pretriangulated dg-category, by Proposition~\ref{prop:b-representable-reedy} 
we have constructed a Reedy cofibrant functor
$B(X, f) : D[n] \to \cat{C}_0$ which is homotopical by Proposition~\ref{prop:b-homotopical},
and thus an $n$-simplex of ${\fNerve(\cat{C}_0)}$.
For $B$ to define a functor between these quasicategories, it has to be
compatible with the simplicial structure maps.

Objects of the slice categories $D[m] \downarrow \alpha$ for some
$\alpha : [k] \to [m]$ are restrictions of $\alpha$ to some subset of $S
\subseteq [k]$, and morphisms are inclusions of these subsets. Hence
postcomposition with some map $\sigma : [m] \to [n]$ induces an isomorphism of
the slice categories $D[m] \downarrow \alpha \to D[m] \downarrow (\sigma \circ \alpha)$.
We use this to show:

\begin{lem}\label{lem:cell-simplicial-iso}
  Let $\alpha : [k] \to [m]$ and $\sigma: [m] \to [n]$. Consider
  $\PathMod(\alpha)$ to be a $\PathAlg(n)$-comodule via base change along the
  map of dg-coalgebras $\PathAlg(m) \to \PathAlg(n)$ induced by $\sigma$.
  Then the map
  $\sigma_* : \PathMod(\alpha) \to \PathMod(\sigma \circ \alpha)$
  induced by $\sigma$ is an isomorphism of $\PathAlg(n)$-comodules.
\end{lem}
\begin{proof}
  The map $\PathMod(\sigma \circ \alpha) \to \PathMod(\alpha)$ that sends a
  restriction $(\sigma \circ \alpha) \restriction_S$ to $\alpha \restriction_S$
  is the inverse chain map of $\sigma_*$. Both comultiplications split
  subsequences of $[k]$ and then apply $\sigma \circ \alpha$ to the suffix,
  which is preserved by both $\sigma_*$ and its inverse.
\end{proof}

In the definition of $B$ we have chosen representatives in $\cat{C}$ for the
representable $\cat{C}$-modules. Assuming that we choose a single representing
object for every isomorphism class of $\cat{C}$-modules, we can show the following:

\begin{prop}\label{prop:b-functor}
  Let $\cat{C}$ be a dg-category. Then the assignment
  $(X, f) \longmapsto B(X, f)$ is a map of simplicial sets
  $B : \dgNerve(\cat{C}) \longrightarrow \fNerve(\cat{C}_0)$,
  and thus a functor of quasicategories.
\end{prop}
\begin{proof}
  Let $(X, f) \in {\dgNerve(\cat{C})}_n$ and let $\sigma : [m] \to [n]$. 
  The map $\sigma$ acts on the diagram $B(X, f) : D[n] \to \cat{C}_0$ by
  precomposition with $D\sigma : D[m] \to D[n]$. The result is represented by
  the diagram $D[m] \to \Mod{\cat{C}}$ given on $\alpha \in D[m]$ by
  \[ \sigma^*(C(X, f))(\alpha) = \prod_{i \in D[n] \downarrow (\sigma \circ \alpha)} \Map{\cat{C}}(X_{(\sigma \circ \alpha \circ i)(0)}, -) \]
  with differential that sends a homogeneous element $\varphi$ to
  \begin{equation}\label{eqn:b-simplicial-action-diff-1}
    d_f\varphi = d\varphi - \sign{|\varphi|} \varphi * f = d\varphi - \sign{|\varphi|} m \circ (\varphi \otimes f) \circ \delta.
  \end{equation}
  Let $\sigma$ act on the $n$-simplex $(X, f)$, then $B(\sigma^*(X, f))$ is represented
  on $\alpha \in D[m]$ by
  \[ C(\sigma^*(X, f))(\alpha) = \prod_{i \in D[n] \downarrow \alpha} \Map{\cat{C}}(X_{(\sigma \circ \alpha \circ i)(0)}, -) \]
  with the differential sending a homogeneous element $\varphi$ to
  \begin{equation}\label{eqn:b-simplicial-action-diff-2}
    d_{\sigma^* f} \varphi =
    d\varphi - \sign{|\varphi|} \varphi * \sigma^* f =
    d\varphi - \sign{|\varphi|} m \circ (\varphi \otimes f) \circ (\id \otimes \PathAlg(\sigma)) \circ \delta.
  \end{equation}
  Observe that $(\id \otimes \PathAlg(\sigma)) \circ \delta$ is the comultiplication
  of $\PathMod(\alpha)$ considered as a $\PathAlg(n)$-comodule.
  The isomorphism $\sigma_* : \PathMod(\alpha) \to \PathMod(\sigma \circ \alpha)$
  induces an isomorphism of graded objects
  \[ \sigma^*(C(X, f))(\alpha) \cong C(\sigma^*(X, f))(\alpha) \]
  by reindexing the product, which is compatible with the differentials
  (\ref{eqn:b-simplicial-action-diff-1}) and (\ref{eqn:b-simplicial-action-diff-2})
  since $\sigma^*$ and its inverse respect the $\PathAlg(n)$-comodule structures by
  Lemma~\ref{lem:cell-simplicial-iso}.
  
  Now it follows that $\sigma^*(B(X, f)) = B(\sigma^*(X, f))$ since they both
  are the representative for the common isomorphism class of $\sigma^*(C(X,
  f))$ and $C(\sigma^*(X, f))$.
\end{proof}

Both $\dgNerve(\cat{C})$ and $\fNerve(\cat{C}_0)$ are functorial in $\cat{C}$, in that
they send dg-functors to functors of quasicategories. The construction of $B$ required
choices of representing objects and universal constructions that in general are not
respected by arbitrary dg-functors, so $B$ is not strictly natural. However, we can show:

\begin{lem}
  Let $\cat{C}$, $\cat{D}$ be pretriangulated dg-categories and
  $F : \cat{C} \to \cat{D}$ a dg-functor. Then
  \[
    \begin{tikzcd}
      \dgNerve(\cat{C}) \ar{r}{B} \ar[swap]{d}{F_*} &
      \fNerve(\cat{C}_0) \ar{d}{F_*} \\
      \dgNerve(\cat{D}) \ar[swap]{r}{B} 
      \ar[Rightarrow, shorten <= 13pt, shorten >= 13pt]{ur}
      &
      \fNerve(\cat{D}_0)
    \end{tikzcd}
  \]
  commutes up to simplicial homotopy.
\end{lem}
\begin{proof}
  Let $(X, f)$ be an $n$-simplex of $\dgNerve(\cat{C})$. We have
  \begin{align}
    &\ \Map{\cat{D}}(B(F_*(X, f))(\alpha), FB(X, f)(\alpha)) \nonumber \\
    \cong&\ C(F_*(X, f))(\alpha)(FB(X, f)(\alpha)) \nonumber \\
    \cong&\ \prod_{i \in D[n] \downarrow \alpha} \Map{\cat{D}}(FX_{(\alpha \circ i)(0)}, F
    B(X, f)(\alpha)) \label{eqn:naturality-repr-f}
  \end{align}
  with differential $d_{F(f)}$, natural in $\alpha \in D[n]$. From the representation
  of the identity map of $B(X, f)(\alpha)$ we get maps $X_{(\alpha \circ i)(0)} \to
  B(X, f)(\alpha)$ for all $i \in D[n] \downarrow \alpha$. We apply the functor $F$
  and obtain an element of (\ref{eqn:naturality-repr-f}), representing a map
  \[ \eta(\alpha) : B(F_*(X, f))(\alpha) \to FB(X, f)(\alpha). \]
  These piece together into a natural transformation
  \[\eta : BF_*(X, f) \to FB(X, f).\]
  For any $\alpha \in D[n]$, we thus have a $1$-simplex $\eta(\alpha)$ of
  $\dgNerve(\cat{D})$, which has a homotopical Reedy cofibrant resolution $B(\eta(\alpha))
  : D[1] \to \cat{D}_0$. Since $B$ constructs compatible resolutions,
  we get a homotopical diagram $D[1] \times D[n] \to \cat{D}_0$, which is Reedy
  cofibrant by an analogous argument as in the proof of
  Proposition~\ref{prop:b-representable-reedy}.
  We then get a homotopical Reedy cofibrant diagram $D[n] \to (\cat{D}_0)_\text{R}^{D[1]}$. Again by the compatibility of the resolutions created by $B$,
  these maps piece together into a map of simplicial sets
  \[\dgNerve(\cat{C}) \longrightarrow \fNerve((\cat{D}_0)_\text{R}^{D[1]}).\]
  By composing this map with the canonical map $\fNerve((\cat{D}_0)_\text{R}^{D[1]}) \to
      \fNerve(\cat{D}_0)^{\Delta[1]}$, we obtain the simplicial homotopy. 
\end{proof}

The maps $\eta(\seq{i})$ for any $i \in [n]$ that occur in the proof above are
isomorphisms, and so by $2$-out-of-$3$ we have that $\eta$ is a natural weak equivalence.
By carefully choosing compatible homotopy inverses and resolutions, it is likely to
be possible to extend the simplicial homotopy to an $E[1]$-homotopy.

\section{Equivalence of Quasicategories}\label{sec:equivalence}

In this last section we will conclude the construction by showing that the
functor $B$ defines an equivalence of quasicategories. Just as in ordinary
category theory, we have to show that $B$ is essentially surjective and
fully faithful~\cite[Definition 1.2.10.1]{higher-topos-theory}.

\begin{prop}\label{prop:b-essentially-surjective}
  Let $\cat{C}$ be a pretriangulated dg-category. Then
  $B : \dgNerve(\cat{C}) \to \fNerve(\cat{C}_0)$
  is essentially surjective as a functor of quasicategories.
\end{prop}
\begin{proof}
  Let $X : D[0] \to \cat{C}_0$ be a homotopical Reedy cofibrant functor
  representing an object of the homotopy category of $\fNerve(\cat{C})$.
  Let $Y := B(X(\seq{0})) $, then 
  $Y: D[0] \to \cat{C}_0$ is Reedy cofibrant by Proposition~\ref{prop:b-representable-reedy}
  and homotopical by Proposition~\ref{prop:b-homotopical} and we have
  we have $Y(\seq{0}) \cong X(\seq{0})$. Thus by Lemma~\ref{lem:frames-equivalent},
  $X$ and $Y$ are equivalent $0$-simplices, and thus isomorphic in the homotopy
  category. Hence $B$ induces an essentially surjective functor on the homotopy
  categories, so by definition it is essentially surjective as a functor of
  quasicategories.
\end{proof}

It remains to show that $B$ is fully faithful. We will use the auxiliary lemma:

\begin{lem}[{\cite[Lemma 4.3]{szumilo-frames}}]\label{lem:weak-equiv-simplicial-sets}
  Let $f : K \to L$ be a map of simplicial sets. Suppose that for each
  $n \in \N$ and a square:
  \[
    \begin{tikzcd}
      \partial \Delta[n] \ar{r}{u} \ar[hookrightarrow]{d} &
      K \ar{d}{f} \\
      \Delta[n] \ar{r}{v} & L
    \end{tikzcd}
  \]
  there are: a map $w: \Delta[n] \to K$ such that $w\restriction_{\partial
    \Delta[n]} = u$ and a homotopy from $f \circ w$ to $v$ relative to the
  boundary. Then $f$ is a weak homotopy equivalence.
\end{lem}

Given the mapping cylinder for some unknown map $f : A \to B$ of chain
complexes, $f$ can be recovered (up to homotopy) by composing the
inclusion $A \to \Cyl(f)$ with a homotopy inverse of 
the inclusion $B \to \Cyl(f)$. Based on this idea, we can prove:

\begin{prop}\label{prop:b-fully-faithful}
  Let $\cat{C}$ be a pretriangulated dg-category. Then 
  \[B: \dgNerve(\cat{C}) \to \fNerve(\cat{C}_0)\]
  is fully faithful as a functor of quasicategories.
\end{prop}
\begin{proof}
  $B$ is fully faithful if it induces a fully faithful functor of homotopy
  categories enriched in the homotopy category of spaces. In particular,
  this is the case when $B$ induces a weak equivalence of simplicial sets
  \[ B_* : \MapR{\dgNerve(\cat{C})}(X, Y) \longrightarrow \MapR{\fNerve(\cat{C}_0)}(B(X), B(Y)) \]
  for all objects $X, Y$, which we prove using Lemma~\ref{lem:weak-equiv-simplicial-sets}:
  Consider a diagram 
  \[
    \begin{tikzcd}
      \partial \Delta[n] \ar{r}{u} \ar[hookrightarrow]{d} &
      \MapR{\dgNerve(\cat{C})}(X, Y) \ar{d}{B_*} \\
      \Delta[n] \ar{r}{v} &
      \MapR{\fNerve(\cat{C}_0)}(B(X), B(Y)).
    \end{tikzcd}
  \]

  To deal with the partially defined objects arising in this proof, we
  implicitly take the pointwise Kan extension of simplicial objects to finite
  simplicial sets. The neccessary colimits exist in $\cat{C}_0$ because of
  Reedy cofibrancy and the closure of $\cat{C}_0$ under pushouts along cofibrations.

  By unpacking the definition of the mapping complex, $u$ corresponds to a map
  \[ \partial \Delta[n] \star \Delta\{ n + 1 \} \longrightarrow \dgNerve(\cat{C}) \]
  which sends $\partial \Delta[n]$ constantly to $X$ and $\Delta\{ n + 1 \}$
  constantly to $Y$. By letting the additional coherence map be zero, we can
  trivially extend this to a map
  \[ U : \Delta[n] \cup (\partial \Delta[n] \star \Delta\{ n + 1 \}) \longrightarrow \dgNerve(\cat{C}), \]
  which is compatible by strict unitality with all extensions that send
  $\Delta[n]$ constantly to $X$ and so get a graded map
  \[ f : \PathAlg(\Delta[n] \cup (\partial \Delta[n] \star \Delta\{ n + 1 \})) \longrightarrow \cat{C} \]
  of degree $-1$.
  On the other hand, we obtain from $v$ a homotopical Reedy cofibrant functor
  \[ V : D([n + 1]) \longrightarrow \cat{C}_0. \]
  By the definition of the mapping space and commutativity of the square we have
  \begin{align*}
    V \restriction D(\Delta[n]) &= B(X) \circ D([n] \to [0]), \\
    V \restriction D(\Delta\{ n + 1 \}) &= B(Y) \circ D(\{ n + 1 \} \to [0]), \\
    V \restriction D(\Delta[n] \cup \partial \Delta[n] \star \Delta\{ n + 1 \}) &= B(U).
  \end{align*}
  In particular, $V$ induces a commutative diagram
  \[
    \begin{tikzcd}
      B(U)(\Delta[n]) \ar{r}{b_1} & V(\Delta[n + 1]) & \ar[swap]{l}{i_1} \ar{dl}{i_2} Y \\
      B(U)(\partial \Delta[n]) \ar[swap]{r}{b_2} \ar{u} & B(U)(\partial \Delta[n] \star \Delta\{ n + 1 \}) \ar{u}{\iota}.
    \end{tikzcd}
  \]
  Let $p_2$ be the homotopy inverse to
  $i_2$ and $H_2 : i_2 \circ p_2 \simeq \id{}$ the homotopy
  that were constructed in Section~\ref{sec:homotopical}.
  $\iota$ is a weak equivalence, because $i_1$ and $i_2$ are weak equivalences, 
  and it is a cofibration by Reedy-cofibrancy of $V$.
  Thus by Lemma~\ref{lem:acyclic-cofibration-retraction} there exists a
  map $p : B \to A$ such that $p \circ \iota = \id$ and a homotopy
  $h : \iota \circ p \simeq \id$ such that $h \circ \iota = 0$. Let
  \[ p_1 = p_2 \circ p, \qquad H_1 = \iota \circ H_2 \circ p + h, \]
  then $H_1 : i_1 \circ p_1 \simeq \id$ and the following diagrams commute:
  \[
    \begin{tikzcd}[column sep = 0.2cm]
      V(\Delta[n + 1]) \ar{r}{p_1} & Y \\
      B(U)(\partial \Delta[n] \star \Delta\{ n + 1 \}) \ar{u}{\iota} \ar[swap]{ur}{p_2}
    \end{tikzcd}
    \;
    \begin{tikzcd}[column sep = 0.2cm]
      V(\Delta[n + 1]) \ar{r}{H_1} & V(\Delta[n + 1]) \\
      B(U)(\partial \Delta[n] \star \Delta\{ n + 1 \}) \ar{u}{\iota} \ar[swap]{r}{H_2} &
      B(U)(\partial \Delta[n] \star \Delta\{ n + 1 \}) \ar{u}{\iota}
    \end{tikzcd}
  \]

  To fill in the interior of $u$, we seek an extension of $f$ to a map
  \[ \hat{f} : \PathAlg(\Delta[n + 1]) \longrightarrow \cat{C}. \]
  The only injective argument $[k] \to [n + 1]$ missing from the domain of $f$
  (and thus the only one whose value is not forced by strict unitality) is the
  identity $[n + 1] \to [n + 1]$, so we seek an element
  $f([n + 1]) \in {\Map{\cat{C}}(X, Y)}_n$. To ensure that
  the extension satisfies the Maurer-Cartan condition,
  the map $f([n + 1])$ must the chosen such that
  \begin{equation}
    \label{eqn:extension-maurer-cartan}
    d(f([n + 1])) = - f(d[n + 1]) - (f * f)([n + 1]).
  \end{equation}

  Consider the diagram of $\cat{C}$-modules
  \[ 
    \begin{tikzcd}[row sep = 1cm, column sep = 0.8cm]
      \Map{\cat{C}}(X, -) & C(U)(\Delta[n]) \ar[swap]{l}{a_1} \ar{d} \ar{dl}{da_1} &
      \Map{\cat{C}}(V(\Delta[n + 1]), -) \ar[swap]{l}{\repr{b_1}} \ar{d}{\repr{\iota}} &
      \Map{\cat{C}}(Y, -) \ar[swap]{l}{\repr{p_1}} \ar{dl}{\repr{p_2}} \\
      \Map{\cat{C}}(X, -) & C(U)(\partial \Delta[n]) \ar{l}{a_2} &
      C(U)(\partial \Delta[n] \cup \Delta\{ n + 1 \}) \ar{l}{\repr{b_2}}
    \end{tikzcd}
  \]
  where $\repr{}$ is the Yoneda embedding and
  $a_1, a_2$ are the maps of degree $n$ and $n - 1$, respectively, 
  which are defined by the following terms:
  \begin{align*}
    a_1(\varphi) &= \sign{|\varphi|n} \varphi([n]), \\
    a_2(\varphi) &= \sign{|\varphi|(n - 1)} \sum_{i = 1}^n \sign{i} \varphi([n] - i) - \sign{|\varphi|(n - 1)} \varphi([n] - 0).
  \end{align*}
  Observe that $a_1, a_2$ are chosen such that the differential of the composition
  of the top row is the composition of the bottom row of the diagram:
  \begin{align*}
    d(a_1 \circ \repr{j_1} \circ \repr{p_1})
    = da_1 \circ \repr{j_1} \circ \repr{p_1}
    = a_2 \circ \repr{j_2} \circ \repr{p_2}.
  \end{align*}
  We can then evaluate the composition of the bottom row for some $\omega \in
  \Map{\cat{C}}(Y, A)$:
  \begin{align*}
    &\ (a_2 \circ \repr{j_2} \circ \repr{p_2})(\omega) \\
    =&\ \sign{(|\omega| + n - 1)(n - 1)} \left( \sum_{i = 1}^n \sign{i} \repr{p_2}([n] - i)(\omega) 
    - \repr{p_2}([n] - 0)(\omega) \right) \\
    =&\ \sign{(|\omega| + n)(n - 1)} \left(
      \sum_{i = 1}^n \sign{i} \omega \circ f([n + 1] - i)
    - \omega \circ f([n + 1] - 0) \right)
    \\
    =&\ \sign{(n + 1)|\omega|} \left(\sum_{i = 1}^n \sign{i} \omega \circ f([n + 1] - i) +
      \omega \circ f([n + 1] - 0) \circ f(\seq{0, 1}) \right)  \\
    =&\ \sign{n}\repr{(f(d[n + 1]) + (f * f)([n + 1]))}(\omega).
  \end{align*}

  Hence the map $\hat{f}([n + 1])$ represented by $- a_1 \circ
  \repr{b_1} \circ \repr{p_1}$ satisfies (\ref{eqn:extension-maurer-cartan}).
  This induces an extension $\hat{u}$ of $u$ that fills in the interior,
  and thus also extends $U$ to
  \[ \hat{U} : \PathAlg(\Delta[n + 1]) \longrightarrow \dgNerve(\cat{C}). \]

  It remains to show that there is a simplicial homotopy from $\hat{u}$ to
  $v$, relative to the boundary. As the essential ingedient for this homotopy,
  we construct a map 
  \[ F([n + 1]) : B(\hat{U})(\Delta[n + 1]) \longrightarrow V(\Delta[n + 1]) \]
  in $\cat{C}_0$. Since $B(\hat{U})(\Delta[n + 1])$ represents $C(\hat{U})(\Delta[n + 1])$,
  such a map amounts to an element
  \[ \varphi \in C(\hat{U})(\Delta[n + 1])(V(\Delta[n + 1])) \]
  of degree zero such that $d_{\hat{f}} \varphi = 0$. Restricted to the proper
  subsets of $[n + 1]$, we let $\varphi$ be determined by the element representing
  the map
  \[ \hat{\iota} : C(U)(\Delta[n] \cup \partial \Delta[n] \star \Delta\{ n + 1 \}) \longrightarrow V(\Delta[n + 1]), \]
  and construct the missing $\varphi([n + 1])$ as follows:
  Consider the diagram of $\cat{C}$-modules
  \[ 
    \begin{tikzcd}[row sep = 1cm, column sep = 0.6cm]
      \Map{\cat{C}}(X, -) & C(U)(\Delta[n]) \ar[swap]{l}{a_1} \ar{d} \ar{dl}{da_1} &
      \Map{\cat{C}}(V(\Delta[n + 1]), -) \ar[swap]{l}{\repr{b_1}} \ar{d}{\repr{\iota}} &
      \Map{\cat{C}}(V(\Delta[n + 1]), -) \ar[swap]{l}{\repr{H_1}} \ar{d}{\repr{\iota}} \\
      \Map{\cat{C}}(X, -) & C(U)(\partial \Delta[n]) \ar{l}{a_2} &
      C(U)(\partial \Delta[n] \star \Delta\{ n + 1 \}) \ar{l}{\repr{b_2}} &
      C(U)(\partial \Delta[n] \star \Delta\{ n + 1 \}) \ar{l}{\repr{H_2}}
    \end{tikzcd}
  \]
  where $a_1, a_2$ are defined as above and calculate the differential of the top row:
  \begin{align*}
    d(a_1 \circ \repr{b_1} \circ \repr{H_1})
    &= a_2 \circ \repr{b_2} \circ \repr{H_2} \circ \repr{\iota} \\
    &+ \sign{n} a_1 \circ \repr{b_1} \circ \repr{p_1} \circ \repr{i_1} \\
    &- \sign{n} a_1 \circ \repr{b_1}.
  \end{align*}
  We evaluate this for some $\omega \in \Map{\cat{C}}(V(\Delta[n + 1]), A)$ and get
  \begin{align}
    d(a_1 \circ \repr{b_1} \circ \repr{H_1})(\omega)
    &= \sign{n|\omega|}\sum_{i = 1}^n \sign{i} \repr{\iota}(\omega)([n + 1] - i)
    \label{eqn:ff-map-1} \\
    & - \sign{n|\omega|} \repr{\iota}(\omega)([n + 1] - 0)
    \label{eqn:ff-map-2} \\
    & + \sign{n|\omega| + n + 1} \repr{i_1}(\omega) \circ \hat{f}([n + 1])
    \label{eqn:ff-map-3} \\
    & + \sign{n|\omega| + n + 1} \repr{b_1}(\omega)([n]).
    \label{eqn:ff-map-4}
  \end{align}

  Then (\ref{eqn:ff-map-1}) and (\ref{eqn:ff-map-4}) together form the sum
  \[ \sum_{i = 1}^{n + 1} \sign{i} \omega \circ \varphi([n + 1] - i). \]
  The terms (\ref{eqn:ff-map-2}) and (\ref{eqn:ff-map-3}) are the only non-vanishing summands of
  \[ \sum_{i = 1}^{n + 1} \sign{ni + n + 1} \omega \circ \varphi(\seq{i, \ldots, n + 1}) \circ \hat{f}(\seq{0, \ldots, i}). \]

  Let $\varphi([n + 1])$ be represented by $a_1 \circ \repr{b_1} \circ \repr{H_1}$,
  then we have shown that
  \[ d(\varphi([n + 1])) = \varphi(d^* [n + 1]) + (\varphi * \hat{f})([n + 1]) \]
  and thus $d_{\hat{f}} \varphi = d\varphi - \varphi * \hat{f} = 0$.
  
  Now we have
  constructed a well-defined map $F([n + 1])$ that is compatible with the inclusion
  maps for proper subsets of $[n + 1]$.
  For some $\alpha : [k] \to [n + 1]$ in $D[n + 1]$ which factors through either $[n]$ or
  $\partial [n] \star \{n + 1\}$, we have $B(\hat{U})(\alpha) = V(\alpha)$, so we
  can pad $F([n + 1])$ with identity maps to obtain a natural transformation
  \[ B(\hat{U}) \restriction S \longrightarrow V \restriction S  \]
  where $S$ is the sieve
  $\Sd[n + 1] \cup D(\Delta[n] \cup \partial \Delta[n] \star \Delta\{ n + 1 \}) \hookrightarrow D[n + 1]$. Consider
  the natural transformation as a homotopical Reedy cofibrant map
  \[ [1] \longrightarrow {(\cat{C}_0^S)}_{\text{R}}. \]
  Precompose with the canonical map $D[1] \to [1]$ to obtain a homotopical,
  yet not neccessarily Reedy cofibrant map
  \[ D[1] \longrightarrow {(\cat{C}_0^S)}_{\text{R}}, \]
  and then use Lemma\ \ref{lem:reedy-modification} to modify this to be Reedy
  cofibrant while leaving the restriction to $[1] \subset D[1]$ fixed.
  By Lemma~\ref{lem:sd-lift} there exists a lift
  \[ 
    \begin{tikzcd}[column sep = 1.5cm]
      \{0, 1\} \ar[hookrightarrow]{d} \ar{r}{B(\hat{U}) \sqcup V} & {(\cat{C}_0^{D[n + 1]})}_{\text{R}} \ar{d} \\
      D[1] \ar{r} \ar[dashed]{ur} & {(\cat{C}_0^S)}_{\text{R}}
    \end{tikzcd},
  \]
  which is also homotopical and Reedy cofibrant. Transpose this to a diagram
  \[ D[1] \times D[n + 1] \longrightarrow \cat{C}_0 \]
  and precompose with the canonical map $D([1] \times [n + 1]) \to D[1] \times D[n + 1]$ to
  get
  \[ D([1] \times [n + 1]) \longrightarrow \cat{C}_0, \]
  which is still homotopical and Reedy cofibrant by Lemma\ \ref{lem:exact-product-d}.
  By construction, this functor agrees with $B(\hat{U})$ when restricted to
  $D(\{ 0 \} \times [n + 1])$, with $V$ when restricted to $D(\{ 1 \} \times [n
  + 1])$, and is the identity on the constant diagrams $B(X)$ and $B(Y)$ on
  $D([1] \times [n])$ and $D([1] \times \{ n + 1\})$, respectively. Hence we obtain
  a simplicial homotopy between $\hat{u}$ and $v$ that leaves the boundary $u$ fixed.
\end{proof}

We can thus conclude with our main result:
\begin{thm}
  Let $\cat{C}$ be a pretriangulated dg-category. Then
  \[ B : \dgNerve(\cat{C}) \longrightarrow \fNerve(\cat{C}_0), \]
  as defined in Section~\ref{sec:construction}, is an equivalence of
  quasicategories.
\end{thm}
\begin{proof}
  $B$ sends simplices of $\dgNerve(\cat{C})$ to Reedy cofibrant diagrams by
  Proposition~\ref{prop:b-representable-reedy}, which are homotopical by
  Proposition~\ref{prop:b-homotopical}. $B$ is a functor of quasicategories by
  Proposition~\ref{prop:b-functor}, which is essentially surjective by
  Proposition~\ref{prop:b-essentially-surjective} and fully faithful by
  Proposition~\ref{prop:b-fully-faithful}. Hence $B$ is an equivalence.
\end{proof}

\bibliographystyle{alpha}
\bibliography{dg-frames}

\end{document}